\documentclass[12pt,a4paper]{amsart}
\usepackage{graphicx,amssymb}
\usepackage{amsmath,amssymb,amscd}
\usepackage[colorinlistoftodos]{todonotes}
\input xy
\xyoption{all}
\usepackage[all]{xy}
\usepackage{hyperref}
\newcommand{\lra}{\longrightarrow}

\newcommand{\RR}{\mathbb R}
\newcommand{\CC}{\mathbb C}

\newcommand{\PP}{\mathbb P}

\newcommand{\cO}{\mathcal{O}}

\newcommand{\Ext}{\operatorname{Ext}}

\newcommand{\End}{\operatorname{End}}
\newcommand{\Hom}{\operatorname{Hom}}

\newcommand{\Gr}{\operatorname{Gr}}
\newcommand{\Aut}{\operatorname{Aut}}
\theoremstyle{plain}
\newtheorem{theorem}{Theorem}[section]
\newtheorem{lem}[theorem]{Lemma}
\newtheorem{prop}[theorem]{Proposition}
\newtheorem{cor}[theorem]{Corollary}
\newtheorem{conj}[theorem]{Conjecture}
\newtheorem{rem}[theorem]{Remark}
\newtheorem{ex}[theorem]{Example}
\numberwithin{equation}{section}

\begin{document}
\title{Coherent systems on the projective line}

\author{P. E. Newstead}
\author{Montserrat Teixidor i Bigas}

\address{P. E. Newstead\\Department of Mathematical Sciences\\
              University of Liverpool\\
              Peach Street, Liverpool L69 7ZL, UK}
\email{newstead@liv.ac.uk}              
\address{M. Teixidor i Bigas\\Mathematics Department\\
              Tufts University\\
              Medford\\
              MA 02155\\
              USA}
\email{montserrat.teixidoribigas@tufts.edu}

\date{\today}

\thanks{Both authors are members of the research group VBAC (Vector Bundles on Algebraic Curves). This work was started while the first author was a Clay Scholar visiting Tufts and Boston Universities.}

\maketitle

\begin{center}{\it Dedicated to the memory of Sir Michael Atiyah}

\end{center}

\begin{abstract}
It is well known that there are no stable bundles of rank greater than 1 on the projective line.
In this paper, our main purpose is to study the existence problem for stable coherent systems on the projective line when the number of sections is larger than the rank. 
We include a review of known results, mostly for a small number of sections.
\end{abstract}

\section{Introduction}\label{intro}

We begin with a tribute to Sir Michael Atiyah by the first author, who says ``I was among the earliest of Sir Michael's students (actually, I think, the fourth) and owe the principal topic of my research, namely vector bundles on algebraic curves, to him. My first supervisor, John Todd, pointed me in the direction of vector bundles but the inspiration for my thesis and a large part of my subsequent work came from Michael. Two of his early papers \cite{at1,at2} were instrumental in this together with the work of M. S. Narasimhan and C. S. Seshadri \cite{ns1,ns2} to which he introduced me at a time when I was still groping for a topic for my thesis. He was a constant support during this period (1964/65) which I spent in Oxford.  He also introduced me to David Mumford, who arranged for me to spend a productive year in Harvard. In large part, I owe my career to him and to an even earlier student of Michael's, Rolph Schwarzenberger, with whom I wrote my first paper. The current paper still falls within the area of vector bundles on algebraic curves.''

Coherent systems on an algebraic curve $C$ of genus $g$ are the higher rank analogue of linear systems and have been the subject of much study in the last 25 years. There is a concept of $\alpha$-stability for coherent systems depending on a real number $\alpha$, which must be positive for $\alpha$-stable coherent systems to exist. Coherent systems are related to higher rank Brill-Noether loci and have gauge-theoretic and symplectic interpretations; they are also closely related to holomorphic $k$-pairs. The survey \cite{bl} describes these relationships and includes relevant references.

The study of coherent systems includes general results valid for any genus (see, for example, \cite{bgmn}). (There are surveys in \cite{bl} and \cite{n}, but these are by no means up to date.) However most papers concentrate on the case $g\ge2$ (see \cite{bgmn} and many other papers, also the survey \cite{n}).  The case $g=1$ is quite straighforward and is studied in \cite{lne} (see also the survey \cite{bl}). 

The case of the projective line has been studied in \cite{ln1,ln2,ln3,bl} and even the existence of $\alpha$-stable coherent systems is quite complicated. The known results include some of a general nature but detailed study has concentrated on the case of coherent systems $(E,V)$ where the dimension $k$ of the subspace $V$ of $H^0(E)$ is less than the rank $n$ of $E$. The paper \cite{ln1} also contains results for  
$k=n$ and $k=n+1$. Our principal object in this paper is to study the existence problem for $k>n$. We also review the known results on the existence of $\alpha$-stable coherent systems on $\PP^1$; 
although not all  of them are necessary for our new results, we have included a complete survey. 

Our main results are described in section \ref{state} following the review in section \ref{pre}. Two conjectures which would come close to completing our results are also stated in section \ref{state}. Section \ref{nec} is concerned with necessary conditions for existence and section \ref{exist1} covers existence for large $\alpha$. These combine to prove Theorems \ref{t1} and \ref{t2}; in particular, Theorem \ref{t1} gives a complete solution in the case when $d$ is a multiple of $n$, while Theorem \ref{t2} gives a complete solution for large $\alpha$ when $d\equiv -1,-2\bmod n$ 
and also when $k\le d$ or $d$ is sufficiently large.
 In section \ref{non}, we obtain a non-existence result which proves Theorem \ref{t3}. Finally, in section \ref{exist2}, we investigate lower bounds on $\alpha$ for the existence of $\alpha$-stable coherent systems, yielding Theorem \ref{t4} and many cases in which we can identify precisely the lower bound, most of which are included in Theorem \ref{t5}.

We work throughout over the field $\CC$. We denote by $\cO$ the trivial line bundle and by $K$ the canonical line bundle on a curve $C$ and by $\cO(1)$ the hyperplane bundle on $\PP^1$. We write $\cO(a):=\cO(1)^{\otimes a}$, reserving the notation $\cO(a)^t$ for the direct sum of $t$ copies of $\cO(a)$. We write also $h^0(E)$ for the dimension of the space of sections $H^0(E)$ of a vector bundle $E$ and $h^1(E)$ for $\dim H^1(E)$.

We thank the referee for some minor corrections, which do not affect the results.

\section{Review of known results}\label{pre}
In this section, we state definitions and review the known results concerning the existence of stable coherent systems on $\PP^1$. These include general results which we shall need in later sections.
\subsection{Definitions}For the purposes of this paper, a \emph{coherent system} of type $(n,d,k)$ is a pair $(E,V)$, where $E$ is a vector bundle of rank $n$ and degree $d$ on a smooth irreducible complex projective curve $C$ of genus $g$ and $V\subset H^0(E)$ is a linear subspace of the space of sections of $E$ of dimension $k$ (for more general definitions, see \cite{LeP,kn}). For general surveys of coherent systems, see \cite{bl,n}.

For any real number $\alpha$, we define the $\alpha$-\emph{slope} $\mu_\alpha(E,V)$ of a coherent system $(E,V)$
of type $(n,d,k)$ by
\[\mu_\alpha(E,V):=\frac{d+\alpha k}n.\]
Then $(E,V)$ is said to be $\alpha$-stable ($\alpha$-semistable) if, for every 
proper subsystem $(E',V')$ of $(E,V)$,
\[\mu_\alpha(E',V')<(\le)\mu_\alpha(E,V).\] 
It follows immediately from the definition that, if $k>0$, $\alpha$-stable bundles can exist only if $d>0$ and $\alpha>0$. There exist moduli spaces $G(\alpha;n,d,k)$ ($\widetilde{G}(\alpha;n,d,k)$) for $\alpha$-stable (S-equivalence classes of $\alpha$-semistable) coherent systems (see \cite{kn}). The moduli space $G(\alpha;n,d,k)$ has an \emph{expected dimension} 
\begin{equation}\label{eq01}
\beta(n,d,k):=n^2(g-1)+1-k(k-d+n(g-1))
\end{equation}
and every irreducible component of $G(\alpha;n,d,k)$ has dimension $\ge\beta(n,d,k)$.
\subsection{Some facts from \cite{bgmn}}
We need some facts from \cite{bgmn} which are not fully covered in \cite{ln1,ln2}. 

We recall first that the groups 
\[\Ext^r((E_2,V_2),(E_1,V_1))\]
exist and that  $\Ext^1((E_2,V_2),(E_1,V_1))$ classifies the equivalence classes of extensions
\[0\lra(E_1,V_1)\lra(E,V)\lra(E_2,V_2)\lra0.\]
Moreover $\Ext^r((E_2,V_2),(E_1,V_1))=0$ for $r\ge3$.
\begin{prop}\label{prop01} \cite[Proposition 3.2]{bgmn}
Let $(E_1,V_1)$ and $(E_2,V_2)$ be coherent systems on $C$ of types $(n_1,d_1,k_1)$ and $(n_2,d_2,k_2)$ respectively. Let 
\[{\mathbb H}^0_{21}:=\Hom((E_2,V_2),(E_1,V_1)),\ \  {\mathbb H}^2_{21}:=\Ext^2((E_2,V_2),(E_1,V_1)).\] 
Then
\begin{equation}\label{eq02}
\dim\Ext^1((E_2,V_2),(E_1,V_1))=C_{21}+\dim{\mathbb H}^0_{21}+\dim{\mathbb H}^2_{21},
\end{equation}
where
\begin{equation}\label{eq03}
C_{21}=n_1n_2(g-1)+d_2n_1-d_1n_2+k_2(d_1-n_1(g-1)-k_1).
\end{equation}
Moreover, if $N_2$ is the kernel of the natural map $V_2\otimes{\mathcal O}\to E_2$, then
\begin{equation}\label{eq04}
{\mathbb H}^2_{21}=H^0(E_1^*\otimes N_2\otimes K)^*.
\end{equation}
\end{prop}

The following fact is also of interest for us.
\begin{prop}\label{prop02} \cite[Proposition 4.5]{bgmn}
Suppose that $k\ge n$. Then the condition of $\alpha$-stability for a coherent system $(E,V)$ of type $(n,d,k)$ is independent of $\alpha$ for $\alpha>d(n-1)$. Moreover
\begin{itemize}
\item if $\frac{k_1}{n_1}<\frac{k}n$ for every proper subsystem $(E_1,V_1)$ of $(E,V)$, then $(E,V)$ is $\alpha$-stable for all $\alpha>d(n-1)$;
\item  if there exists a subsystem $(E_1,V_1)$ of $(E,V)$ such that $\frac{k_1}{n_1}>\frac{k}n$, then $(E,V)$ is not $\alpha$-semistable for $\alpha>d(n-1)$.
\end{itemize}
\end{prop}

As a consequence of Proposition \ref{prop02}, there are only finitely many critical points at which the stability of a coherent system of type $(n,d,k)$ can change. This means that there are only finitely many distinct moduli spaces $G(\alpha;n,d,k)$. 
For any $\alpha'$, we define $G(\alpha'^+;n,d,k)$ to be the moduli space $G(\alpha;n,d,k)$ for $\alpha$ in a suitable open interval $]\alpha',\alpha''[$ (and similarly  $G(\alpha^-;n,d,k)$ in an open interval $]\alpha'',\alpha'[$).
 When $k<n$, $G(\alpha;n,d,k)=\emptyset$ if $\alpha\ge\frac{d}{n-k}$ \cite[Lemma 4.2]{bgmn}. If $k\ge n$, there is no automatic upper bound for $\alpha$, but $G(\alpha;n,d,k)$ remains constant for sufficiently large $\alpha$. 

\subsection{General results for coherent systems on $\PP^1$} The results in this subsection are taken principally from \cite{ln1} and \cite{ln2}, and include all the results on existence  of $\alpha$-stable coherent systems from those papers. From now on, we will be working with bundles on $\PP^1$.

The main fact that distinguishes the theory of vector bundles on $\PP^1$ from that on curves of higher genus is that there exist no stable bundles of rank $\ge2$. In fact, every vector bundle $E$ on $\PP^1$ can be written uniquely as
\[E\cong\bigoplus_{i=1}^n\cO(a_i)\mbox{ with }a_1\ge a_2\ge\cdots\ge a_n.\]
In modern form, this is due to Grothendieck, but it is essentially a statement in linear algebra and as such goes back at least to Kronecker. Note also that $h^0(\cO(a))=a+1$. We say that $E$ is of \emph{generic splitting type} if $a_1\le a_n+1$, in other words $E$ can be written as
\[E\cong\cO(a)^{n-t}\oplus\cO(a-1)^t\]
for some integers $a$, $t$ with $0\le t<n$.
A bundle $E$ is of generic splitting type if and only if $h^1(\End(E))=0$.

\begin{lem}\label{l01}\cite[Lemma 3.1]{ln1}\cite[Lemma 3.1 and Corollary 3.2]{ln2}
Suppose $k>0$ and $(E,V)$ is $\alpha$-stable ($\alpha$-semistable) for some $\alpha>0$.  Then
\[E\cong\bigoplus_{i=1}^n\cO(a_i)\]
with all $a_i>0$ ($a_i\ge0$). Moreover, if $(E_1,V_1)$ and $(E_2,V_2)$ are $\alpha$-semistable (possibly for different values of $\alpha$), then
\[\Ext^2((E_1,V_1),(E_2,V_2))=\Ext^2((E_2,V_2),(E_1,V_1))=0.\]
\end{lem}

\begin{theorem}\label{t01}\cite[Theorem 3.2, Corollaries 3.3 and 3.4]{ln1}
Suppose $k>0$. 
If $G(\alpha;n,d,k)\ne\emptyset$, then $\beta(n,d,k)\ge0$ and $G(\alpha;n,d,k)$ is smooth and irreducible and has 
dimension $\beta(n,d,k)$.
Moreover
\begin{itemize}
\item[(i)] for a general $(E,V)\in G(\alpha;n,d,k)$, $E$ is of generic splitting type and $V$ is generic in $\Gr(k,H^0(E))$;
\item[(ii)] for any $(n,d,k)$, the set $\{\alpha|G(\alpha;n,d,k)\ne\emptyset\}$ is an open interval 
$I(n,d,k)$ in $\RR$ (possibly empty or semi-infinite).
\end{itemize}
\end{theorem}

\begin{rem}\begin{em}
The genericity of $V$ is not stated in \cite{ln1}; it depends on the openness of the $\alpha$-stability condition, which is the basis for the proof of \cite[Theorem 3.2]{ln1}.
\end{em}\end{rem}

\begin{prop}\label{prop03}\cite[Propositions 3.5, 3.6 and 3.7]{ln1}
Suppose that $G(\alpha;n,d,k)\ne\emptyset$. Then, for a general $(E,V)\in G(\alpha;n,d,k)$,
\begin{itemize}
\item if $0<k<n$, we have an exact sequence
\[0\lra\cO^k\lra E\lra G\lra0,\]
where $V=H^0(\cO^k)\subset H^0(E)$ and $G$ is a vector bundle; moreover
\[E\cong\cO(a)^{n-t}\oplus\cO(a-1)^t\mbox{ and }G\cong\cO(a+l+1)^m\oplus\cO(a+l)^{n-k-m},\]
where the integers $a$, $t$, $l$ and $m$ are defined by
\[d=an-t\mbox{ with }0\le t<n\mbox{ and }ka-t=l(n-k)+m\mbox{ with }0\le m<n-k;\]
\item if $k=n$, we have an exact sequence
\[0\lra\cO^n\lra E\lra T\lra0,\]
where $V=H^0(\cO^n)\subset H^0(E)$ and $T$ is a torsion sheaf of length $d$;
\item if $k>n$, we have an exact sequence
\[0\lra H\lra V\otimes \cO\lra E\lra0;\]
moreover $E\cong\cO(a)^{n-t}\oplus\cO(a-1)^t$, where $d=an-t$ with $0\le t<n$.
\end{itemize}
\end{prop}

\begin{prop}\label{prop04}\cite[Propositions 4.1 and 4.2]{ln1}
Suppose $k>0$ and $G(\alpha;n,an-t,k)\ne\emptyset$, where $0\le t<n$. Then $\alpha>\frac{t}k$. If, in addition, $k<n$ and $l$, $m$ are defined as in Proposition \ref{prop03}, then
\[\frac{t}k<\alpha<\frac{d}{n-k}-\frac{mn}{k(n-k)}.\]
\end{prop}

\begin{lem}\label{lem02}\cite[Lemma 3.9]{ln2}
Suppose that $(F,W)$ is $\alpha$-stable with $\mu_\alpha(F,W)=b$ and that 
\begin{equation}\label{eq05}
\dim\Ext^1((F,W),(\cO(b),0))\ge r
\end{equation}
Then the general extension
\[0\lra(\cO(b)^r,0)\lra(E,V)\lra(F,W)\lra0\]
is $\alpha^+$-stable.
\end{lem}

\subsection{Special cases} There are some further results from \cite{ln1,ln2} which are of interest in the context of existence of $\alpha$-stable coherent systems on $\PP^1$. The results of \cite{ln3} do not concern us, as they relate to the structure of the moduli spaces rather than to the existence problem. 
\begin{theorem}\label{t02}\cite[Theorem 4.5]{ln2}
Suppose $0<k<n$. Then $G(0^+;n,an,k)\ne\emptyset$ if and only if $k((a+1)n-k)\ge n^2-1$.
\end{theorem}

\begin{prop}\label{prop05}\cite[Corollaries 5.9 and 5.10]{ln2}
Let $k$ be a fixed positive integer and let $l$ be defined as in Proposition \ref{prop03}. Then, for all but finitely many pairs $(n,d)$ with $n>k$, one of the following two possibilities holds:
\begin{itemize}
\item $G(\alpha;n,d,k)=\emptyset$ for all $\alpha>0$;
\item $G(\alpha;n,d,k)\ne\emptyset$ for all $\alpha$ such that $\frac{t}k<\alpha<\frac{ln+t}k$.
\end{itemize}
Moreover, for all but finitely many pairs $(n,a)$ with $n>k$ and $k((a+1)n-k)\ge n^2-1$, the moduli space $G(\alpha;n,an,k)\ne\emptyset$ if and only if $0<\alpha<\frac{ln}k$.
\end{prop}

\begin{prop}\label{prop06}\cite[Proposition 5.12]{ln2}
Suppose $k\ge2$, $(a+1)k\ge n+t$ and that one of the following three conditions holds:
\begin{itemize}
\item $t=1$ and $a\ge k$;
\item $t=k-1$ and $a\ge2$;
\item $t=k$ and $a\ge3$.
\end{itemize}
Then $G((t/k)^+;n,an-t,k)\ne\emptyset$.
\end{prop}

Results were also obtained in \cite{ln1,ln2} for certain specific values of $k$.

\begin{theorem}\label{t03}\cite[Theorems 5.1 and 5.4]{ln1}\cite[Theorem 8.4]{ln2}
Let $d=an-t$ with $0\le t<n$ and let $l$ and $m$ be defined as in Proposition \ref{prop03}.
\begin{itemize}
\item If $n\ge2$, $G(\alpha;n,d,1)\ne\emptyset$ for some $\alpha>0$ if and only if $l\ge1$; moreover, when $l\ge1$, $G(\alpha;n,d,1)\ne\emptyset$ if and only if $t<\alpha<\frac{d-mn}{n-1}$.
\item If $n\ge3$, $G(\alpha;n,d,2)\ne\emptyset$ for some $\alpha>0$ if and only if $l\ge1$, $d\ge\frac{n(n-2)+3}2$ and $(n,d)\ne(4,6)$; moreover, when these conditions hold, $G(\alpha;n,d,2)\ne\emptyset$ if and only if  $\frac{t}2<\alpha<\frac{2d-mn}{2(n-2)}$.
\item If $n\ge4$, $G(\alpha;n,d,3)\ne\emptyset$ for some $\alpha>0$ if and only if $l\ge1$, $d\ge\frac{n(n-3)+8}3$ and $(n,d)\ne(6,9)$; moreover, when these conditions hold, $G(\alpha;n,d,3)\ne\emptyset$ if and only if $\frac{t}3<\alpha<\frac{3d-mn}{3(n-3)}$, except for the following pairs $(n,d)$, where $I:=I(n,d,3)$ is as stated: for $(4,7)$: $I=]\frac35,7[$; for $(5,9)$, $I=]\frac34,\frac{11}3[$; for $(6,11)$: $I=]1,\frac73[$; for $(7,13)$: $I=]\frac32,\frac83[$.
\end{itemize}
\end{theorem}

The following theorem completes the results for $k=2$ and $k=3$.

\begin{theorem}\label{t04}\cite[Proposition 5.6]{ln1}\cite[Theorem 9.2]{ln2}
Let $d=an-t$ with $0\le t<n$.
\begin{itemize}
\item $G(\alpha;2,d,2)\ne\emptyset$ if and only if $d>2$ and $\alpha>\frac{t}2$.
\item $G(\alpha;2,d,3)\ne\emptyset$ for some $\alpha$ if and only if $d\ge2$. Moreover, if $d\ge2$, then $G(\alpha;2,d,3)\ne\emptyset$ for all $\alpha>\frac{t}3$ except in the case $d=3$, when $G(\alpha;2,3,3)\ne\emptyset$ if and only if $\alpha>1$.
\item $G(\alpha;3,d,3)\ne\emptyset$ for some $\alpha$ if and only if $d\ge4$. Moreover, if $d\ge4$, then $G(\alpha;3,d,3)\ne\emptyset$ for all $\alpha>\frac{t}3$ except in the case $d=5$, when $G(\alpha;3,5,3)\ne\emptyset$ if and only if $\alpha>\frac23$.
\end{itemize}
\end{theorem}

The paper \cite{ln2} also contains an example for $k=4$ with unexpected behaviour.

\begin{prop}\label{prop07}\cite[Proposition 10.1]{ln2}
$G(\alpha;6,7,4)\ne\emptyset$ if and only if $\frac54<\alpha<2$.
\end{prop}
\begin{rem}\label{r01}\begin{em}The expected range of $\alpha$ here is $\frac54<\alpha<\frac{11}4$.
\end{em}\end{rem}

Finally, we have non-emptiness results for $k=n-1$, $k=n$ and $k=n+1$.

\begin{prop}\label{prop08}\cite[Propositions 6.1, 6.3 and 6.4]{ln1}
\begin{itemize}
\item $G(\alpha;n,d,n-1)\ne\emptyset$ for some $\alpha$ if and only if $d\ge n$. Moreover, in this case, the upper bound for $\alpha$ is precisely $d$.
\item $G(\alpha;n,d,n)\ne\emptyset$ for some $\alpha$ if and only if $d>n$. Moreover, in this case, there is no upper bound on $\alpha$. 
\item $G(\alpha;n,d,n+1)\ne\emptyset$ for some $\alpha$ if and only if $d\ge n$. Moreover, if $d\ge n$ and we write $d=an-t$ with $0\le t<n$, then $G(\alpha;n,d,n+1)$ is always non-empty if $\alpha>t$.
\end{itemize}
\end{prop}

\section{Statement of results and conjectures}\label{state}
Our first theorem gives a complete answer to the existence problem for $\alpha$-stable coherent systems when $d$ is a multiple of $n$ and $k>n$.
\begin{theorem}\label{t1}
Let $n,a,k$ be positive integers with $n\ge2$, $k>n$ and let $d=an$. Then
\begin{itemize}
\item for all $\alpha>0$, the moduli space $G(\alpha;n,d,k)$ is non-empty if and only if $k((a+1)n-k)\ge n^2-1$;
\item $G(\alpha;n,d,k)$ is smooth and irreducible of dimension $k((a+1)n-k)- n^2+1$ whenever it is non-empty;
\item if  $k((a+1)n-k)\ge n^2-1$, 
$$ U^s:=\{(E,V)|(E,V)\mbox{ is of type }(n,d,k) \mbox{ and  is }\alpha\mbox{-stable for all }\alpha>0\}$$
is a smooth and irreducible variety of dimension $k(n(a+1)-k)- n^2+1$.
\end{itemize}
\end{theorem}

When $d$ is not a multiple of $n$, we have the following theorem.

\begin{theorem}\label{t2}
Let $n,a,k$ be positive integers with $n\ge2$, $k\ge n$ and let $d=an-t$ with $1\le t\le n-1$. Suppose that $G(\alpha;n,d,k)$ is non-empty for some $\alpha>0$. Then
\begin{itemize}
\item $k((a+1)n-t-k)\ge n^2-1$;
\item $G(\alpha;n,d,k)$ is smooth and irreducible of dimension $k((a+1)n-t-k)- n^2+1$ whenever it is non-empty;
\item $k<an$;
\item $\alpha>\alpha_c:=\max\{\frac{t}k,\frac{n-t}{an-k}\}$.
\end{itemize}
Moreover, if either $k\le an-t$ or $k\le an-1$ and $a\ge t$, then $G(\alpha;n,d,k)$ is non-empty for sufficiently large $\alpha$.
\end{theorem}

This theorem answers the existence problem for large $\alpha$ if $t=1$ or $t=2$ but not if $t\ge3$. The next theorem provides some non-existence results and shows that the necessary conditions of Theorem \ref{t2} are not sufficient for non-emptiness. In particular, if $n\ge4$, $G(\alpha;n,2n-3,2n-1)=\emptyset$ for all $\alpha>0$.
Note, however, that the theorem is compatible with Theorem \ref{t2}.

\begin{theorem}\label{t3}
Suppose that $a\ge2$, $1\le t\le n-1$ and one of the following holds:
\begin{itemize}
\item[(a)] $k\ge at$ and $(a-1)t>a(an-k)+(a-2)n$;
\item[(b)] $k\le at$ and $t> (a+1)k-n$.
\end{itemize} 
Then $G(\alpha;n,an-t,k)=\emptyset$ for all $\alpha>0$.
\end{theorem}

Our final major theorem provides cases in which the lower bound for $\alpha$ in Theorem \ref{t2} is sharp.
Note that the hypotheses in Theorem \ref{t3}(b) can never occur if $k\ge n$. In fact, Theorem \ref{t4}(b) is a generalisation of Proposition \ref{prop06}. 

\begin{theorem}\label{t4}
Suppose that $a\ge2$, $1\le t\le n-1$ and one of the following holds:
\begin{itemize}
\item[(a)] $at\le k<an$, $(a-1)t\le a(an-k)+(a-2)n$ and $G(\alpha_c;n-t,a(n-t),k-at)\ne\emptyset$;
\item[(b)] $k\le at$, $t\le(a+1)k-n$ and $G(\alpha_c;t,(a-1)t,k)\ne\emptyset$.
\end{itemize}
Then $G(\alpha_c^+;n,an-t,k)\ne\emptyset$.
\end{theorem}

\begin{rem}\label{r02}\begin{em}When $a=2$, $k=t\ge2$ and $n=2t$, Theorem \ref{t4} gives no information since $G(\alpha;t,t,t)=\emptyset$ for all $\alpha$ by Proposition \ref{prop08}. In fact, $G(\alpha;2t,3t,t)=\emptyset$ for all $\alpha$ (see Example \ref{ex2}).
\end{em}\end{rem}

In view of Theorems \ref{t2} and \ref{t3} and Remark \ref{r02}, and noting that, if $k\ge n$, the inequality $t\le(a+1)k-n$ holds automatically, it seems reasonable to make the following conjectures.

\begin{conj}\label{conj1}
Suppose that $a\ge2$, $1\le t\le n-1$, $k\ge n$ and $k((a+1)n-t-k)\ge n^2-1$. Then $G(\alpha;n,an-t,k)$ is non-empty for some $\alpha>0$ if and only if one of the following holds:
\begin{itemize}
\item[(a)] $at< k<an$ and $(a-1)t\le a(an-k)+(a-2)n$;
\item[(b)] $k\le at$.
\end{itemize}
\end{conj}

\begin{conj}\label{conj2}
Suppose that $a\ge2$, $1\le t\le n-1$ and $k\ge n$. If $G(\alpha^0;n,an-t,k)$ is non-empty for some $\alpha^0>0$, then $G(\alpha;n,an-t,k)$ is non-empty for all $\alpha\ge\alpha^0$.
\end{conj}

Possible counterexamples to Conjecture \ref{conj2} are indicated in Example \ref{ex1}.
When $G(\alpha^0;n,an-t,k)$ is non-empty for some $\alpha^0>0$, we make no conjecture about the precise lowerbound $\alpha_{lb}$ for the set of $\alpha$ for which $G(\alpha;n,an-t,k)\ne\emptyset$. However, we do have $\alpha_c\le\alpha_{lb}\le d(n-1)$
by Theorem \ref{t2} and Proposition \ref{prop02} and, if $G(\alpha;n,an-t,k)\ne\emptyset$ for large $\alpha$, then  $\alpha_{lb}\le\frac{t(n-t)}h$, where $h=\gcd(n,k)$ (see Corollary \ref{c0}). On the other hand, we often have $\alpha_{lb}=\alpha_c$ and $G(\alpha;n,an-t,k)\ne\emptyset$ for all $\alpha>\alpha_c$. The following theorem includes many cases in which this is guaranteed.
\begin{theorem}\label{t5} Suppose that $a\ge2$, $1\le t\le n-1$ and $k\ge n$. Then $G(\alpha;n,an-t,k)\ne\emptyset$ for all $\alpha>\alpha_c$ if one of the following holds:
\begin{itemize}
\item[(a)] $(a-1)t+n\le k\le an-t$;
\item[(b)] $(a-1)t+n\le k<an$, $a\ge t$, $(a-1)t\le a(an-k)+(a-2)n$;
\item[(c)] $k=at+1$, $a\ge \max\{n-t-1,t\}$,  $(a-1)t\le a(an-k)+(a-2)n$;
\item[(d)] $k<at$, $k(at-k)\ge t^2-1$.
\end{itemize}
\end{theorem}

There are other cases for which the same result holds, including
\begin{itemize}
\item $a\ge2$, $t=1$, $n\le k\le a$ or $n+a-1\le k<an$ (Proposition \ref{prop8}); note that this covers the case $n=2$ completely;
\item $n\ge3$, $a\ge n-1$, $t=n-1$, $a(n-1)<k<an$ (Corollary \ref{c5}).
\end{itemize}

\section{Necessary conditions}\label{nec}

In this section, we prove the following proposition, where we do not assume that $k\ge n$. Note that $\beta(n,d,k)=k(d+n-k)-n^2+1$ (see \eqref{eq01}).

\begin{prop}\label{prop1}
Let $n$, $a$, $k$ be positive integers with $n\ge2$ and let $d=an-t$ with $0\le t\le n-1$. Suppose that
$G(\alpha;n,d,k)\ne\emptyset$ for some $\alpha>0$. Then
\begin{itemize}
\item $k((a+1)n-t-k)\ge n^2-1$;
\item $G(\alpha;n,d,k)$ is smooth and irreducible of dimension $k((a+1)n-t-k)- n^2+1$; 
\item either $t=0$ or 
\begin{equation}\label{eq00}
t>0,\ k<an\ \mbox{and}\ \alpha>\max\left\{\frac{t}k,\frac{n-t}{an-k}\right\}.
\end{equation}
\end{itemize}\end{prop}
\begin{proof}
In  view of Theorem \ref{t01} and the definitions, we need only prove that, if $t>0$, then \eqref{eq00} holds.
Suppose then that $t>0$. By Theorem \ref{t01} again, the generic element of $G(\alpha;n,d,k)$ is of the form $(E,V)$ with $E\cong{\mathcal O}(a)^{n-t}\oplus {\mathcal O}(a-1)^t$. 
Consider the subsystem $(E',V')$, where $E'={\mathcal O}(a)^{n-t}$ and $V'=\emptyset$. 
The stability condition then gives 
$$a< a-\frac{t}n+\alpha\frac{k}n,$$
hence $\alpha>\frac{t}k$. 
Consider the subsystem $(E',V')$, where $E'={\mathcal O}(a)^{n-t}$ and $V'=V\cap H^0(E')$. Note that
$$\dim V'\ge k-(h^0(E)-h^0(E'))=k-at.$$
If $k>at$, the condition gives
$$a+\alpha\frac{k-at}{n-t}<a-\frac{t}n+\alpha\frac{k}n,$$
which simplifies to
$$\alpha(an-k)>n-t.$$
This implies that $an-k>0$ and then $\alpha>\frac{n-t}{an-k}$, completing the proof. 
(For the bound $\alpha>\frac{t}k$, see also Proposition \ref{prop04}.)
\end{proof}

\section{Existence for large $\alpha$}\label{exist1}

In this section we obtain existence results for large $\alpha$. This enables us in particular to complete the proofs of Theorems \ref{t1} and \ref{t2}. Our first result is the following proposition.

\begin{prop}\label{prop2}
Suppose that $n\ge2$, $a\ge2$, $0\le t\le n-1$ and $E={\mathcal O}(a)^{n-t}\oplus {\mathcal O}(a-1)^t$. Suppose further that $n<k\le an-t$. Then there exists a subspace $V$ of $H^0(E)$ such that $(E,V)$ admits no subsystem $(E_1,V_1)$ of type $(n_1,d_1,k_1)$ with 
\begin{equation}\label{eq41}\frac{k_1}{n_1}\ge\frac{k}n.
\end{equation}
\end{prop}

We shall prove this proposition by showing that the dimension of the family of coherent systems $(E_1,V_1)$ which satisfy \eqref{eq41} is less than $\dim\Gr(k,h^0(E))$. Note that, if $E_1={\mathcal O}(b)\oplus E_1'$ with $b\le0$, then we can replace $(E_1,V_1)$ by $(E_1',V_1\cap H^0(E_1'))$ of type $(n_1-1,d_1',k_1')$ with $k_1'\ge k_1-1$ and \eqref{eq41} 
together with the condition $k\ge n$
imply that $\frac{k_1'}{n_1-1}\ge\frac{k}n$. Thus we can assume that every direct factor of $E_1$ has degree $\ge1$; combining this with other obvious inequalities, we have
\begin{equation}\label{eq42}
n_1\le d_1\le \min\{an_1,(a-1)n_1+n-t\},\ \ k_1\le h^0(E_1)= d_1+n_1.
\end{equation}
We can now count the number of parameters on which the choice of $V$ depends.

In the first place, if we fix $n_1$ and $d_1$, the fact that the degree of any rank one direct factor of $E_1$ lies between $1$ and $a$ means there are finitely many choices for $E_1$. For each of these, the inclusion of $E_1$ as a subbundle of $E$ depends on at most 
$$\dim\Hom(E_1,E)-\dim\Aut(E_1)\le n((a+1)n_1-d_1)-tn_1-n_1^2$$
parameters by Riemann-Roch. Next, the choice of $V_1\subset H^0(E_1)$ depends on at most 
$$\dim\Gr(k_1,d_1+n_1)=k_1(d_1+n_1-k_1)$$
parameters. Finally, the choice of $V\subset H^0(E)$ containing $V_1$ corresponds to a point in $Gr(k-k_1, H^0(E)/V_1)$ and depends on
$$(k-k_1)((a+1)n-t-k_1-(k-k_1))=(k-k_1)((a+1)n-t-k)$$
parameters. We need to prove that the sum of these numbers is strictly less than
$$\dim\Gr(k,(a+1)n-t)=k((a+1)n-t-k);$$
in other words, we need to prove that the hypotheses of Proposition \ref{prop2} imply
\begin{equation}\label{eq43}
n((a+1)n_1-d_1)-tn_1-n_1^2+k_1(d_1+n_1-(a+1)n+t+k-k_1)<0.
\end{equation}

\begin{lem}\label{lem1}
Suppose that \eqref{eq41} and \eqref{eq42} hold and either $t=0$, $k<(a+1)n$ or $t>0$, $k<an$. Then
\begin{equation}\label{eq44}
d_1+n_1-(a+1)n+t+k-k_1<0.
\end{equation}\end{lem}
\begin{proof}
Suppose first that $t=0$. Using \eqref{eq41} , \eqref{eq42} and the assumption that $k<(a+1)n$, we have
$$k-k_1\le k-\frac{n_1k}n=\frac{(n-n_1)k}n<(a+1)(n-n_1)\le(a+1)n-d_1-n_1,$$
which gives \eqref{eq44}.

If $t>0$, we obtain similarly

$$k-k_1\le\frac{(n-n_1)k}n<a(n-n_1)\le an-d_1+n-n_1-t$$
by \eqref{eq41}, \eqref{eq42} and the assumption that $k<an$. This gives \eqref{eq44}.
\end{proof}

Given \eqref{eq44}, it follows from \eqref{eq41}
that \eqref{eq43} holds if
\begin{equation}\label{eq45}
n((a+1)n_1-d_1)-tn_1-n_1^2+\frac{n_1k}n\left(d_1+n_1-(a+1)n+t+k-\frac{n_1k}n\right)<0.
\end{equation}

\begin{lem}\label{lem2}
Suppose that $a\ge2$, $0\le t\le n-1$, $n<k\le an-t$ and \eqref{eq41} and \eqref{eq42} hold. Then \eqref{eq45} holds.
\end{lem}

\begin{proof}
For fixed $n_1$ and $d_1$, the left hand side of \eqref{eq45} is a quadratic in $k$ with positive leading coefficient. It is therefore sufficient to prove that non-strict inequality holds in \eqref{eq45} for $k=n$ and strict inequality for $k=an-t$. 

For $k=n$, the left hand side of \eqref{eq45} becomes
$$n((a+1)n_1-d_1)-tn_1-n_1^2+n_1(d_1-an+t)=(n-n_1)(n_1-d_1)\le0$$
by \eqref{eq42}.

Now suppose $k=an-t$. The formula \eqref{eq45} becomes
\begin{equation}\label{eq46}
n((a+1)n_1-d_1)-tn_1-n_1^2+\left(an_1-\frac{tn_1}n\right)\left(d_1-(a-1)n_1-n+\frac{tn_1}n\right)<0.
\end{equation}
By \eqref{eq42}, it is sufficient to prove this for $n_1\le d_1\le (a-1)n_1+n-t$.
Since \eqref{eq46} is linear in $d_1$, it is in fact sufficient to prove it for $d_1=n_1$ and for $d_1=(a-1)n_1+n-t$.

For $d_1=n_1$, the left hand side of \eqref{eq46} becomes
$$-(a-1)^2n_1^2+\frac{2(a-1)tn_1^2}n-\frac{t^2n_1^2}{n^2}=-n_1^2\left(a-1-\frac{t}n\right)^2<0.$$

Now suppose that $d_1=(a-1)n_1+n-t$. In this case, \eqref{eq46} becomes
$$n(2n_1-n+t)-tn_1-n_1^2+\left(an_1-\frac{tn_1}n\right)\left(-t+\frac{tn_1}n\right)<0.$$
Simplifying, this becomes
$$(n-n_1)\left(-n+n_1+t-\frac{atn_1}n+\frac{t^2n_1}{n^2}\right)<0.$$
We therefore need to prove that 
\begin{equation}\label{eqnew}
-n+n_1+t-\frac{atn_1}n+\frac{t^2n_1}{n^2}<0.
\end{equation}
If $t=0$, this is clear. If $t>0$, it is sufficient to prove \eqref{eqnew} for $a=2$ with $n_1=1$ and $n_1=n-1$. In fact, for $a=2$, $n_1=1$, we obtain
$$-n+1+t-\frac{2t}n+\frac{t^2}{n^2}\le\frac{t}n\left(\frac{t}n-2\right)<0.$$
For $a=2$, $n_1=n-1$, we have
$$-1+t-\frac{2t(n-1)}n+\frac{t^2(n-1)}{n^2}=-\left(1-\frac{t}n\right)^2-t +\frac{t^2}n<0.$$
\end{proof}

\begin{proof}[Proof of Proposition \ref{prop2}]
The proposition follows from Lemmas \ref{lem1} and \ref{lem2}
\end{proof}

In the case $t=0$, we can extend Proposition \ref{prop2} to cover the range $an<k<(a+1)n$. We begin with the following lemma.

\begin{lem}\label{lem3}
Suppose that $t=0$, $k>n$, $k((a+1)n-k)\ge n^2-1$ and $d_1=an_1$. Then \eqref{eq41} and \eqref{eq42} imply \eqref{eq43}.
\end{lem}
\begin{proof}
Since the hypotheses of Lemma \ref{lem1} hold, it is sufficient to prove \eqref{eq45}, which now takes the form

$$
nn_1-n_1^2+\frac{n_1k}n\left((a+1)(n_1-n)+k-\frac{n_1k}n\right)<0.
$$
Rearranging this, we obtain 
\begin{equation}\label{eq47}
a+1>\frac{k}n+\frac{n}k.
\end{equation}

Recall now the hypothesis $k((a+1)n-k)\ge n^2-1$, which is equivalent to
$$a+1\ge \frac{k}n+\frac{n}k-\frac1{nk}.$$
It follows that the only circumstances under which \eqref{eq47} fails are
$$a+1=\frac{k}n+\frac{n}k\ \ \mbox{and}\ \ a+1=\frac{k}n+\frac{n}k-\frac1{nk}.$$
In the first case, it is easy to see, by considering the prime factors of $n$ and $k$, that $k=n$, which contradicts the assumption that $k>n$. In the second case, it is easy to see that $n$ and $k$ are coprime. Since $n_1<n$, it follows from \eqref{eq41} that
$$k_1\ge\frac{n_1k+1}n.$$
Using this inequality in place of \eqref{eq41} and substituting again in \eqref{eq43}, it is now sufficient to prove that 
$$nn_1-n_1^2+\frac{n_1k+1}n\left((a+1)(n_1-n)+k-\frac{n_1k+1}n\right)<0.$$
Equivalently
$$a+1>\frac{k}n-\frac1{n(n-n_1)}+\frac{nn_1}{n_1k+1}.$$
On substituting $a+1=\frac{k}n+\frac{n}k-\frac1{nk}$, this becomes
$$\frac{n}k+\frac1{n(n-n_1)}>\frac{nn_1}{n_1k+1}+\frac1{nk},$$
which is obviously true.
\end{proof}

\begin{prop}\label{prop3}
Suppose that $n\ge2$, $a\ge2$ and $E={\mathcal O}(a)^n$. If $k((a+1)n-k)\ge n^2-1$ and $k>an$, then there exists a subspace $V$ of $H^0(E)$ such that $(E,V)$ admits no proper subsystem $(E_1,V_1)$ of type $(n_1,d_1,k_1)$ with 
$\frac{k_1}{n_1}\ge\frac{k}n$.
\end{prop}
\begin{proof}
Suppose first that $n_1=1$. Then, by \eqref{eq41}, we have $k_1\ge\frac{k}n>a$. Hence, by \eqref{eq42}, we have $k_1=a+1$ and $d_1=a$. The result now follows from Lemma \ref{lem3}.

We now proceed by induction on $n_1$, assuming that $n_1\ge2$ and the result is proved for rank $n_1-1$. If $d_1=an_1$, the result follows from Lemma \ref{lem3}. Otherwise, we can write $E_1={\mathcal O}(b)\oplus E_1'$ with $b<a$ and put $V_1':=V_1\cap H^0(E_1')$. Then $h^0(E_1')=h^0(E_1)-(b+1)$ and so, using \eqref{eq41},
$$k_1':=\dim V_1'\ge k_1-(b+1)\ge k_1-a\ge \frac{kn_1}n-a=\frac{k(n_1-1)}n+\frac{k}n-a>\frac{k(n_1-1)}n.$$
Now apply the inductive hypothesis to $(E_1',V_1')$.
\end{proof}

\begin{proof}[Proof of Theorem \ref{t1}]
The only point that remains to be proved is that, if $k((a+1)n-k)\ge n^2-1$, then $U^s$ is non-empty. This follows from Propositions \ref{prop02}, \ref{prop1}, \ref{prop2} and \ref{prop3} and the fact that $E={\mathcal O}(a)^{n}$ is semistable.
\end{proof}

\begin{rem}\label{r1}\begin{em}
When $t=0$, the proof of Lemma \ref{lem2} and hence also that of Proposition \ref{prop2} fail when $k=n$. In fact $G(\alpha;n,an,n)$ is non-empty for all $\alpha>0$ if $a\ge2$, but $G(\alpha;n,n,n)=\emptyset$ for all $\alpha$ (see Proposition \ref{prop08}).
\end{em}\end{rem}

In order to obtain further information when $t\ge2$, suppose first that $k=an-1$. Then \eqref{eq45} becomes
\begin{equation}\label{eq48}
n((a+1)n_1-d_1)-tn_1-n_1^2+\left(an_1-\frac{n_1}n\right)\left(d_1-(a-1)n_1-n+t-1+\frac{n_1}n\right)<0.
\end{equation}
For $d_1=n_1$, we have, after dividing by $n_1$, expanding and collecting terms,
\[\left(a-1-\frac1n\right)\left(t-(a-1)n_1-1+\frac{n_1}n\right)<0.\]
Since $a-1-\frac1n>0$, this simplifies to $t<(a-1)n_1+1-\frac{n_1}n$ or, equivalently,
\begin{equation}\label{eq49}
t\le(a-1)n_1.
\end{equation}
For $d_1=(a-1)n_1+n-t$, we have $t<n-n_1+\frac{an_1}n-\frac{n_1}{n^2}$ or, equivalently
\begin{equation}\label{eq50}
t\le n-n_1 +\frac{an_1-1}n.
\end{equation}

\begin{lem}\label{lem4}
Suppose that $1\le t\le n-1$, $1\le n_1\le n-1$, $n< k\le an-1$ and \eqref{eq41} and \eqref{eq42} hold. Suppose further that \eqref{eq49} and \eqref{eq50} hold. Then \eqref{eq45} holds for all allowable values of $d_1$.
\end{lem}
\begin{proof}
Following the proof of Lemma \ref{lem2}, it is sufficient to prove \eqref{eq45} for $k=an-1$ and $n_1\le d_1\le (a-1)n_1+n-t)$, in other words to prove \eqref{eq48} for all such $d_1$. But this is true if and only if \eqref{eq49} and \eqref{eq50} hold.
\end{proof}

\begin{prop}\label{prop4}
Suppose that $n\ge2$, $a\ge2$ and $E={\mathcal O}(a)^{n-t}\oplus{\mathcal O}(a-1)^t$ with $2\le t\le n-1$ and $a\ge t$. Suppose further that $n< k\le an-1$.
Then there exists a subspace $V$ of $H^0(E)$ such that $(E,V)$ admits no subsystem  $(E_1,V_1)$ of type $(n_1,d_1,k_1)$ with $\frac{k_1}{n_1}\ge\frac{k}n$.
\end{prop}

\begin{proof}
In view of Lemma \ref{lem4}, it is sufficient to check \eqref{eq49} and \eqref{eq50} for $1\le n_1\le n-1$. For $a\ge t$, using that $t\le n-1$,
 it is easy to see that \eqref{eq50} is always true. On the other hand, \eqref{eq49} is true if $a>t$; it is also true if $a=t\ge2$ and $n_1\ge2$. It remains to consider the case $a=t\ge2$, $n_1=1$.

In this case, in view of Lemma \ref{lem2} and \eqref{eq42}, we can restrict to the range 
$$an-t<k\le an-1,\ \ a-1\le d_1\le a$$
and we need to verify \eqref{eq48}, which now becomes
$$n(a+1-d_1)-a-1+\left(a-\frac1n\right)\left(d_1-n+\frac1n\right)<0.$$
For $d_1=a$, this becomes 
$$(a-1)(a-n)-\frac1{n^2}<0,$$
which is true since $2\le a\le n-1$.
For $d_1=a-1$, we require
$$(a-2)(a-n)+\frac1n-\frac1{n^2}<0,$$
which holds for $3\le a\le n-1$. We are left with the case $a=t=2$, $d_1=n_1=1$.

From \eqref{eq41} and \eqref{eq42}, we now have $k_1=2$ 
and we can check directly that \eqref{eq43} holds.
This completes the proof.
\end{proof}

\begin{proof}[Proof of Theorem \ref{t2}]
The necessary conditions for non-emptiness of $G(\alpha;n,d,k)$ stated in the theorem follow at once from Proposition \ref{prop1}. If $k=n$ and $k<an$, we have $an-t>n$, so $G(\alpha;n,d,k)\ne\emptyset$ for large $\alpha$ by Proposition \ref{prop08}. If $n<k\le an-t$, then $G(\alpha;n,d,k)\ne\emptyset$ for large $\alpha$ by Propositions \ref{prop02} and \ref{prop2}. Finally, if $n<k\le an-1$, $a\ge t$, the same holds by Propositions \ref{prop02} and \ref{prop4}.
\end{proof}

\section{A non-existence result}\label{non}

In this section, we prove Theorem \ref{t3}. The first proposition shows in particular that the necessary conditions of Theorem \ref{t2} are not sufficient for the existence of $\alpha$-stable coherent systems (see Corollary \ref{c6}).

\begin{prop}\label{prop5}
Suppose that $a\ge2$, $1\le t\le n-1$, $k\ge at$ and 
\begin{equation}\label{eq54}
(a-1)t> a(an-k)+(a-2)n.
\end{equation} Then $G(\alpha;n,an-t,k)=\emptyset$ for all $\alpha>0$.
\end{prop}
\begin{proof}Suppose that $G(\alpha;n,an-t,k)\ne\emptyset$ and let $(E,V)$ be a general element of $G(\alpha;n,an-t,k)$. Then $E\cong{\mathcal O}(a)^{n-t}\oplus{\mathcal O(a-1)}^t$. Moreover, by Theorem \ref{t01}(i), we can take $V$ to be a general subspace of $H^0(E)$, so that the homomorphism $V\to H^0({\mathcal O}(a-1)^t)$ is surjective.
It follows that there is an exact sequence
\begin{equation}\label{eq53}
0\lra ({\mathcal O}(a)^{n-t},W)\lra(E,V)\lra({\mathcal O}(a-1), H^0({\mathcal O}(a-1)))^t\lra0
\end{equation}
with $\dim W=k-at$. Sequences \eqref{eq53} are classified by a $t$-tuple of elements in 
$$\Ext^1(({\mathcal O}(a-1), H^0({\mathcal O}(a-1)),({\mathcal O}(a)^{n-t},W)).$$
For $(E,V)$ to be $\alpha$-stable, these elements must be linearly independent. On the other hand, since $(E,V)$ is $\alpha$-stable, we have 
$$\Hom(({\mathcal O}(a-1), H^0({\mathcal O}(a-1))),({\mathcal O}(a)^{n-t},W))=0.$$ 
Moreover, by \eqref{eq04}, 
$$\Ext^2(({\mathcal O}(a-1), H^0({\mathcal O}(a-1))),({\mathcal O}(a)^{n-t},W))=0.$$ 
Hence, by \eqref{eq02} and \eqref{eq03},
\begin{multline}\nonumber\dim(\Ext^1(({\mathcal O}(a-1), H^0({\mathcal O}(a-1))),({\mathcal O}(a)^{n-t},W)))\\=-(n-t)+(a-1)(n-t)-a(n-t)+a(a(n-t)+n-t-k+at)\\=(a-2)(n-t)+a(an-k)< t.\end{multline}
This is a contradiction, which proves the proposition.
\end{proof}

\begin{cor}\label{c6} Suppose that $n\ge4$. Then $G(\alpha;n,2n-3,2n-1)=\emptyset$.
\end{cor}
\begin{proof} Take $a=2$, $t=3$, $k=2n-1$ in the proposition.
\end{proof}

\begin{prop}\label{prop9} Suppose that $a\ge2$, $1\le t\le n-1$, $k\le at$ and $t>(a+1)k-n$. Then $G(\alpha;n,an-t,k)=\emptyset$ for all $\alpha>0$.
\end{prop}

\begin{proof} The proof is a modification of that of Proposition \ref{prop5}. Again we suppose that $(E,V)$ is a general element of $G(\alpha;n,an-t,k)$. Then, by Theorem \ref{t01}(i),
$E\cong{\mathcal O}(a)^{n-t}\oplus{\mathcal O(a-1)}^t$ and $V\cap H^0({\mathcal O}(a)^{n-t})=0$. It follows that there is an exact sequence
\begin{equation}\label{eq55}
0\lra ({\mathcal O}(a)^{n-t},0)\lra(E,V)\lra({\mathcal O}(a-1)^t, V')\lra0
\end{equation}
with $\dim V'=k$. Sequences \eqref{eq55} are classified by a $(n-t)$-tuple of elements in 
$$\Ext^1(({\mathcal O}(a-1)^t, V'),({\mathcal O}(a),0)).$$
For $(E,V)$ to be $\alpha$-stable, these elements must be linearly independent. Since $(E,V)$ is $\alpha$-stable, we have 
$$\Hom(({\mathcal O}(a-1)^t, V'),({\mathcal O}(a),0))=0.$$ 
Moreover, by \eqref{eq04}, 
$$\Ext^2(({\mathcal O}(a-1)^t, V'),({\mathcal O}(a),0))=0.$$ 
Hence, by \eqref{eq02} and \eqref{eq03},
\begin{multline}\nonumber\dim(\Ext^1(({\mathcal O}(a-1)^t, V'),({\mathcal O}(a),0)))=-t+(a-1)t-at+k(a+1)\\=-2t+k(a+1)< n-t.\end{multline}
This is a contradiction, which proves the proposition.
\end{proof}

\begin{proof}[Proof of Theorem \ref{t3}]
The theorem is a combination of Propositions \ref{prop5} and \ref{prop9}.
\end{proof}

\section{Existence for small $\alpha$}\label{exist2}

In this section, we look for lower bounds on $\alpha$ for the non-emptiness of $G(\alpha;n,d,k)$ and prove Theorem \ref{t4}. 

By Proposition \ref{prop02}, the $\alpha$-stability condition is independent of $\alpha$ for $\alpha>d(n-1)$. However we can do much better than this.

\begin{prop}\label{prop6} Suppose that $a\ge2$ and $1\le t\le n-1$ and let $h:=\gcd(n,k)$. Suppose further that $(E,V)\in G(\alpha;n,an-t,k)$ for large $\alpha$ with $E\cong{\mathcal O}(a)^{n-t}\oplus{\mathcal O(a-1)}^t$. Then $(E,V)$ is $\alpha$-stable for $\alpha>\frac{t(n-t)}h$.
\end{prop}
\begin{proof}
Suppose that $(E,V)$ is not $\alpha$-stable for some $\alpha>0$. By Proposition \ref{prop02}, any proper subsystem
 $(E_1,V_1)$ of type $(n_1,d_1,k_1)$ has either $\frac{k_1}{n_1}<\frac{k}{n}$ or $\frac{k_1}{n_1}=\frac{k}{n}$ and $\frac{d_1}{n_1}<\frac{d}{n}$. Since $(E,V)$ is not $\alpha$-stable for some $\alpha$, it follows that there exists $(E_1,V_1)$ with $\frac{k_1}{n_1}<\frac{k}{n}$. This subsystem contradicts $\alpha$-stability if and only if
$$\frac{d_1}{n_1}+\alpha\frac{k_1}{n_1}\ge\frac{d}n+\alpha\frac{k}n.$$
Note that $n_1k-nk_1\ge h$ and $d_1\le\min\{an_1,(a-1)n_1+n-t\}$. So
$$\alpha\left(\frac{k}n-\frac{k_1}{n_1}\right)\le \frac{d_1}{n_1}-\frac{d}n\le \frac1{n_1}\min\{an_1,(a-1)n_1+n-t\}-\frac{an-t}n,$$
and
$$\alpha h\le \alpha(n_1k-nk_1)\le\min\{n_1t,n_1t+n(n-t-n_1)\}.$$
For this to hold for some $n_1$, we require $\alpha\le \frac{t(n-t)}h$. It follows that $(E,V)$ is $\alpha$-stable for $\alpha>\frac{t(n-t)}h$.
\end{proof}

\begin{rem}\label{r5}\begin{em} In view of Propositions \ref{prop02} and \ref{prop6} and Theorem \ref{t2}, if $t>0$, the open interval $I(n,d,k):=\{\alpha|G(\alpha;n,d,k)\ne\emptyset\}$ is either empty or takes one of the following forms:
\begin{itemize}
\item $I(n,d,k)=]\alpha_1,\alpha_2[$ with $\alpha_c\le\alpha_1<\alpha_2\le d(n-1)$;
\item $I(n,d,k)=]\alpha_1,\infty[$ with $\alpha_c\le\alpha_1\le\frac{t(n-t)}h$.
\end{itemize}
\end{em}\end{rem}

\begin{cor}\label{c0}
If $G(\alpha;n,an-t,k)$ is non-empty for large $\alpha$, it is non-empty for all $\alpha>\frac{t(n-t)}h$.
\end{cor}
\begin{proof}This follows from Proposition \ref{prop6} since the general element of $G(\alpha;n,an-t,k)$ has the form $(E,V)$ with $E\cong{\mathcal O}(a)^{n-t}\oplus{\mathcal O(a-1)}^t$.
\end{proof}

This estimate is in general a long way removed from the necessary condition 
$$\alpha>\max\left\{\frac{t}k,\frac{n-t}{an-k}\right\}$$
of Theorem \ref{t2}. However, there are cases in which it is best possible, for example $G(\alpha;2,3,3)$ (see Theorem \ref{t04}). More generally, we have

\begin{cor}\label{c1} Suppose that $a\ge2$ and $h$ divides $n$. Then 
$G(\alpha;n,an-1,an-h)$ is non-empty if and only if $\alpha>\frac{n-1}h$.
\end{cor}
\begin{proof} This follows at once from Theorems \ref{t01}(ii) and \ref{t2} and Corollary \ref{c0}.
\end{proof}

\begin{prop}\label{prop7}
Suppose that $a\ge2$, $1\le t\le n-1$, $at\le k<an$ and
\begin{equation}\label{eq61}(a-1)t\le a(an-k)+(a-2)n.
\end{equation} 
Let $\alpha_c:=\frac{n-t}{an-k}$ and suppose that $G(\alpha_c;n-t,a(n-t),k-at)\ne\emptyset$. Then 
\begin{itemize}
\item[(i)]$G(\alpha;n,an-t,k)=\emptyset$ for $\alpha\le\alpha_c$;
\item[(ii)]$G(\alpha_c^+;n,an-t,k)\ne\emptyset$.
\end{itemize}
\end{prop}
\begin{proof} (i) This is an immediate consequence of Proposition \ref{prop1}.

(ii) The general element of $G(\alpha_c;n-t,a(n-t),k-at)\ne\emptyset$ has the form $({\mathcal O}(a)^{n-t},W)$, where $\dim W=k-at$. We now consider again sequences \eqref{eq53}. The same calculation as in the proof of Proposition \ref{prop5} gives
$$\dim(\Ext^1(({\mathcal O}(a-1), H^0({\mathcal O}(a-1))),({\mathcal O}(a)^{n-t},W)))\ge t.$$
We can therefore choose the elements of the $t$-tuple classifying \eqref{eq53} to be linearly independent. Now note that 
$$\mu_{\alpha_c}({\mathcal O}(a)^{n-t},W)=\mu_{\alpha_c}(({\mathcal O}(a-1), H^0({\mathcal O}(a-1)))^t),$$
so $(E,V)$ is $\alpha_c$-semistable. It follows that any subsystem $(E_1,V_1)$ contradicting $\alpha_c^+$-stability will also contradict $\alpha_c$-stability and therefore has the same $\alpha_c$-slope as $({\mathcal O}(a-1), H^0({\mathcal O}(a-1)))^t$ and $({\mathcal O}(a)^{n-t},W)$. This can happen only if $(E_1,V_1)$ either contains $({\mathcal O}(a)^{n-t},W)$ or intersects it in $0$ and also maps onto a direct factor of $({\mathcal O}(a-1), H^0({\mathcal O}(a-1)))^t$. If the intersection is $0$, this contradicts the linear independence condition. Otherwise, we have an exact sequence
\[0\lra(\cO(a)^{n-t},W)\lra(E_1,V_1)\lra(\cO(a-1),H^0((a-1)))^s\lra0 \]
for some $s<t$. We have
\begin{eqnarray*}\mu_\alpha(E_1,V_1)&=&\frac1{n-t+s}\big(a(n-t)+\alpha(k-at)+(a-1)s+\alpha as\big)\\
&=&a+\frac1{n-t+s}\big(-s+\alpha(k-at+as)\big).
\end{eqnarray*}
So $(E_1,V_1)$ contradicts the $\alpha$-stability of $(E,V)$ if and only if
\[a+\frac1{n-t+s}\big(-s+\alpha(k-at+as)\big)\ge\mu_\alpha(E,V)=a+\frac1n(-t+\alpha k),\]
i.e., if and only if
\[n(-s+\alpha(k-at+as))\ge(n-t+s)(-t+\alpha k).\]
Rearranging, this becomes
\[\alpha(t-s)(an-k)\le(t-s)(n-t),\]
i.e., 
\[\alpha\le\alpha_c.\]
So $(E,V)$ is $\alpha_c^+$-stable.
\end{proof}

\begin{cor}\label{c2}
Suppose that the hypotheses of Proposition \ref{prop7} hold with $k\ge n$ and either $k\le an-t$ or $a\ge t$. Then
\begin{itemize}
\item[(i)] $G(\alpha;n,an-t,k)\ne\emptyset$ if and only if $\alpha>\alpha_c$;
\item[(ii)] $G(\alpha;n,2n-1,n+1)\ne\emptyset$ if and only if $\alpha>1$.
\end{itemize}
\end{cor}
\begin{proof} (i) This follows at once from Theorems \ref{t01}(ii) and \ref{t2} and Proposition \ref{prop7}.

(ii) This is a special case of (i), noting that now $\alpha_c=1$.
\end{proof}

\begin{rem}\label{r2}\begin{em}
(i) The special case $n=2$, $k=3$ of Corollary \ref{c2} is included in Theorem \ref{t04}.

(ii) There is no assumption in Proposition \ref{prop7} that $k\ge n$.

(iii) The hypothesis $G(\alpha_c;n-t,a(n-t),k-at)\ne\emptyset$ certainly holds by Theorem \ref{t1} and Remark \ref{r1}  if $k-at\ge n-t$, i.e. $k\ge n+(a-1)t$. It may still  hold for smaller values of $k$ but not for $k=at$ unless $t=n-1$.
\end{em}\end{rem}

\begin{cor}\label{c5}
Suppose that $n\ge2$, $a\ge n-1$, $t=n-1$ and $a(n-1)<k<an$. Then $G(\alpha;n,(a-1)n+1,k)\ne\emptyset$ if and only if $\alpha>\alpha_c=\frac1{an-k}$.
\end{cor}
\begin{proof}
Note that the hypotheses imply that $a\ge2$. In this case \eqref{eq61} holds and $G(\alpha;1,a,k-a(n-1))\ne\emptyset$ for all $\alpha$. The result follows from Theorems \ref{t01}(ii) and \ref{t2} and Proposition \ref{prop7}.
\end{proof}

\begin{ex}\label{ex1}\begin{em}
The assumption $a\ge n-1$ in Corollary \ref{c5} is needed in order to apply Theorem \ref{t2} but not for Proposition \ref{prop7}. Suppose in particular that $n\ge3$, $a\ge2$, $t=n-1$ and $k=a(n-1)$. Then \eqref{eq61} holds if and only if $n\le a^2+a-1$. So, by Proposition \ref{prop5}, $G(\alpha;n,(a-1)n+1,a(n-1))=\emptyset$ if $n>a^2+a-1$. On the other hand, we have certainly $G(\alpha;1,a,0)\ne\emptyset$ for all $\alpha$, so, if $n\le a^2+a-1$, then $G(\alpha_c^+;n,(a-1)n+1,a(n-1))\ne\emptyset$ by Proposition \ref{prop7}. On the other hand, we know from Theorem \ref{t2} that $G(\alpha;n,(a-1)n+1,a(n-1))\ne\emptyset$ for large $\alpha$ if $n\le a+1$, but we do not know whether this still holds when $a+1<n\le a^2+a-1$. This therefore provides candidates for counterexamples to Conjecture \ref{conj2}.
\end{em}\end{ex}

\begin{cor}\label{c3}
Suppose that \eqref{eq61} holds and $k=at+1\ge n$ with $a\ge2$, $1\le t\le n-1$ and $a\ge \max\{n-t-1,t\}$. Then $G(\alpha;n,d,k)\ne\emptyset$ if and only if $\alpha>\alpha_c=\frac{n-t}{a(n-t)-1}$.
\end{cor}
\begin{proof}
We need to show that $G(\alpha_c;n-t,a(n-t),1)\ne\emptyset$. When $t=n-1$, $G(\alpha;1,a,1)\ne\emptyset$ for all $\alpha$. Suppose now that $t\le n-2$. For coherent systems of this type, the numbers $l$ and $m$ in Proposition \ref{prop03} are defined by $a=l(n-t-1)+m$ with $0\le m<n-t-1$. Since $a\ge n-t-1$, we have $l\ge1$, so, by Theorem \ref{t03}, $G(\alpha_c;n-t,a(n-t),1)\ne\emptyset$ provided that
$$\alpha_c=\frac{n-t}{a(n-t)-1}<\frac{a(n-t)-m(n-t)}{n-t-1}.$$
This is easily seen to be true. The result now follows from Theorems \ref{t01}(ii) and \ref{t2} and Proposition \ref{prop7}.
\end{proof} 

In the case $t=1$, we can improve Proposition \ref{prop7}.

\begin{prop}\label{prop8}
Suppose that $n\ge2$, $a\ge2$ and $k<an$ and let $\alpha_c:=\max\{\frac1k,\frac{n-1}{an-k}\}$. Then $G(\alpha_c^+;n,an-1,k)\ne\emptyset$ in the following three cases:
\begin{itemize}
\item[(a)] $k\le a$ and $(a+1)k\ge n+1$;
\item[(b)] $k\ge n+a-1$;
\item[(c)] $a<k<n+a-1$ and $G(\alpha_c;n-1,a(n-1),k-a)\ne\emptyset$.
\end{itemize}
If, in addition, $k\ge n$, then $G(\alpha;n,an-1,k)\ne\emptyset$ if and only if $\alpha>\alpha_c$.
\end{prop}
\begin{proof}
(a) This is contained in Proposition \ref{prop06}.

(b) Here \eqref{eq61} holds, so this is a consequence of Proposition \ref{prop7} and Remark \ref{r2}.

(c) This is a special case of Proposition \ref{prop7}.

The final assertion now follows from Theorem \ref{t2}.
\end{proof}

\begin{cor}\label{c4}
Suppose that $a\ge2$ and $2\le k<2a$. Then $G(\alpha;2,2a-1,k)\ne\emptyset$ if and only if $\alpha>\max\{\frac1k,\frac1{2a-k}\}$.
\end{cor}
\begin{proof} In this case, either (a) or (b) holds.
\end{proof}

\begin{rem}\label{r3}\begin{em}
For $k=2$ and $k=3$, see also Theorem \ref{t04}. 
\end{em}\end{rem}

\begin{prop}\label{prop10}
Suppose that $a\ge2$, $1\le t\le n-1$, $k\le at$ and
\begin{equation}\label{eq72}
t\le(a+1)k-n.
\end{equation}
Let $\alpha_c=\frac{t}k$ and suppose that $G(\alpha_c;t,(a-1)t,k)\ne\emptyset$. Then 
\begin{itemize}
\item[(i)] $G(\alpha;n,an-t,k)=\emptyset$ for $\alpha\le\alpha_c$;
\item[(ii)] $G(\alpha_c^+;n,an-t,k)\ne\emptyset$.
\end{itemize}\end{prop}
\begin{proof} (i) This follows from Proposition  \ref{prop1}.

(ii) We consider sequences \eqref{eq55}. The same calculation as in the proof of Proposition \ref{prop9} shows that 
$$\dim(\Ext^1(({\mathcal O}(a-1)^t, V'),({\mathcal O}(a),0))\ge n-t.$$
The result follows from Lemma \ref{lem02}.
\end{proof}

\begin{ex}\label{ex2}\begin{em}
In the case $a=2$, $k=t\ge2$, Proposition \ref{prop10} gives no information since $G(\alpha;t,t,t)=\emptyset$ for all $\alpha$ by Proposition \ref{prop08}. In fact, $G(\alpha;n,2n-t,t)=\emptyset$ for all $\alpha$ if $n\ge 2t$. If $n>2t$, this a special case of Proposition  \ref{prop9}. For $n=2t$, we can use the argument of \cite[Remark 8.3]{ln2}; the cases $t=2$ and $t=3$ are already covered in Theorem \ref{t03}.  In general, if $G(\alpha;2t,3t,t)\ne\emptyset$, its general element has the form $(E,V)$ with $E\cong\cO(2)^t\oplus\cO(1)^t$ and there is an exact sequence
\[0\lra(\cO(2)^t,0)\lra(E,V)\lra(\cO(1)^t,W)\lra0\]
with $\dim W=t$. Since the evaluation map $W\otimes\cO\lra\cO(1)^t$ is not an isomorphism, there exists a section of $\cO(1)^t$ contained in $W$ which has a zero. Hence $(\cO(1)^t,W)$ has a subsystem $(\cO(1),W_1)$ with $\dim W_1=1$ and we have an exact sequence
\begin{equation}\label{eq73}
0\lra(\cO(2)^t,0)\lra(E_1,V_1)\lra(\cO(1),W_1)\lra0.
\end{equation}
By \eqref{eq02}, \eqref{eq03} and Lemma \ref{l01}, we have
\[\Ext^1((\cO(1),W_1),(\cO(2),0)=1,\]
so the sequence \eqref{eq73} is induced from a sequence
\[0\lra(\cO(2),0)\lra(E_2,V_2)\lra(\cO(1),W_1)\lra0.\]
Now $(E_2,V_2)$ is a subsystem of $(E,V)$ which contradicts $\alpha$-stability for all $\alpha$.
\end{em}\end{ex}

\begin{cor}\label{c7}
Suppose that the hypotheses of Proposition \ref{prop10} hold and that $k\ge n$ and $k(at-k)\ge t^2-1$. Then $G(\alpha;n,an-t,k)\ne\emptyset$ if and only if $\alpha>\alpha_c$.
\end{cor}
\begin{proof}
This follows from Theorems \ref{t01}(ii) and \ref{t2} and Proposition \ref{prop10}
\end{proof} 

\begin{proof}[Proof of Theorem \ref{t4}]
The theorem is a combination of Propositions \ref{prop7} and \ref{prop10}.
\end{proof}

\begin{proof}[Proof of Theorem \ref{t5}]
(a) and (b) follow from Corollary \ref{c2} once we have verified the hypotheses of Proposition \ref{prop7}. In fact, if $k\le at$, \eqref{eq61} holds automatically, so it can be omitted from (a). It remains to show that $G(\alpha_c;n-t,a(n-t),k-at)\ne\emptyset$. This holds by Theorem \ref{t1} and Remark \ref{r1} since $k-at\ge n-t$ by hypothesis.

(c) is Corollary \ref{c3}.

(d) follows from Corollary \ref{c7} once we have verified the hypotheses of Proposition \ref{prop10}. The inequality \eqref{eq72} holds automatically since $k\ge n$, while $G(\alpha_c;t,(a-1)t,k)\ne\emptyset$  by Theorem \ref{t1} since $k>t$ and $k(at-k)\ge t^2-1$ by hypothesis.
\end{proof}

The following remark gives conditions for non-emptiness for small $\alpha$ which are not covered by any of the above results.
\begin{rem}\label{r4}\begin{em}
Suppose $a\ge2$. By Proposition \ref{prop08}, $G(\alpha;n,an-t,n+1)\ne\emptyset$ for all $\alpha>t$. This is best possible for $t=0$ but not always for $t\ge1$ (see, for example, Corollary \ref{c4} with $a=3$ and $k=3$).
\end{em}\end{rem}

\end{document}